\documentclass[12pt]{amsart}

\usepackage{graphicx}
\usepackage{amsmath}

\def\si{\sigma}

\newcommand{\der}{{\rm d}}

\numberwithin{equation}{section}

\newtheorem{theorem}{Theorem}[section]

\theoremstyle{remark}

\theoremstyle{remark}

\usepackage{amssymb,stmaryrd}
\usepackage{amscd}

\author{Matthew Randall}

\email{mran052@gmail.com}

\title{Automorphism of solutions to Ramanujan's differential equations and other results}

\subjclass[2010]{34M15} 

\begin{document}

\begin{abstract}
In part one we prove a theorem about the automorphism of solutions to Ramanujan's differential equations. We also investigate possible applications of the result. In part two we prove a similar theorem about the automorphism of solutions to the first-order system of differential equations associated to the generalised Chazy equation with parameter $k=\frac{3}{2}$.
\end{abstract}

\maketitle

\pagestyle{myheadings}
\markboth{Randall}{Automorphism of solutions to Ramanujan's differential equations}

\section{Part One}
We say that the triple of functions $(p(x),q(x),r(x))$ of the variable $x$ satisfies Ramanujan's differential equations if the following set of equations are satisfied for the functions $p(x)$, $q(x)$ and $r(x)$ in the triple:
\begin{align}\label{rde0}
\frac{\der p}{\der x}&=\frac{1}{6}(p^2-q),\nonumber\\
\frac{\der q}{\der x}&=\frac{2}{3}(pq-r),\\
\frac{\der r}{\der x}&=pr-q^2.\nonumber
\end{align}

We prove the following theorem:
\begin{theorem}\label{thm1}
Suppose $(P(x),Q(x),R(x))$ satisfies Ramanujan's differential equations, i.e. we have
\begin{align}\label{rde}
\frac{\der}{\der x}P&=\frac{1}{6}(P^2-Q),\nonumber\\
\frac{\der}{\der x}Q&=\frac{2}{3}(PQ-R),\\
\frac{\der}{\der x}R&=PR-Q^2.\nonumber
\end{align}
Let  $T=R+\sqrt{R^2-Q^3}$ and consider the quantities
\begin{align*}
 v=\frac{3}{2}T^{\frac{1}{3}}+\frac{3}{2}\frac{Q}{T^{\frac{1}{3}}},\hspace{12pt} u=\pm \sqrt{3}(\frac{Q^2}{T^{\frac{2}{3}}}+Q+T^{\frac{2}{3}})^{\frac{1}{2}}.
 \end{align*}
Then the following holds. The triples
\begin{align*}
 &(p_2,q_2,r_2)\\
 &=\left(P+\frac{u+v}{2},\frac{8}{9}u(u+v)+\frac{1}{36}(v-u)^2,\frac{1}{54}(3u+v)(16u(u+v)-(\frac{v-u}{2})^2)\right),\\
 &(p_3,q_3,r_3)\\
 &=\left(P+\frac{v-u}{2},\frac{8}{9}u(u-v)+\frac{1}{36}(v+u)^2,\frac{1}{54}(v-3u)(16u(u-v)-(\frac{u+v}{2})^2)\right),
 \end{align*}
also satisfy Ramanujan's differential equations (\ref{rde0}), and furthermore so does the triple
\begin{align*}
(p_0,q_0,r_0)&=\left(P+\frac{1}{2}T^{\frac{1}{3}}+\frac{1}{2}\frac{Q}{T^{\frac{1}{3}}},~ \frac{3}{2}Q+\frac{5}{4}T^{\frac{2}{3}}+\frac{5}{4}\frac{Q^2}{T^{\frac{2}{3}}},~\frac{11}{4}R+\frac{21}{8}QT^{\frac{1}{3}}+\frac{21}{8}\frac{Q^2}{T^{\frac{1}{3}}}\right).
\end{align*} 
\end{theorem}

\begin{proof}
The proof relies on some known facts about Chazy's equation, a third-order nonlinear ODE that is equivalent to Ramanujan's system (\ref{rde0}). 
It is well known that Chazy's equation 
\begin{equation}\label{chazy}
y'''-2yy''+3(y')^2=0
\end{equation}
is equivalent to the Darboux-Halphen system 
\begin{align}\label{dh}
w_1'&=w_2w_3-w_1w_2-w_1w_3\nonumber,\\
w_2'&=w_3w_1-w_2w_3-w_2w_1,\\
w_3'&=w_1w_2-w_3w_1-w_3w_2,\nonumber
\end{align}
with $y=-2(w_1+w_2+w_3)$. See \cite{chazypara} for further details about this equivalence.
Here
\begin{align*}
w_1&=-\frac{1}{2}\frac{\der}{\der x}\log\frac{s'}{s (s-1)},\\
w_2&=-\frac{1}{2}\frac{\der}{\der x}\log\frac{s'}{s-1},\\
w_3&=-\frac{1}{2}\frac{\der}{\der x}\log\frac{s'}{s},
\end{align*}
and $s$ is the Schwarz triangle function $s(0,0,0,x)$ given by solutions of the third-order differential equation
\[
\{s,x\}+\frac{(s')^2}{2}\left(\frac{1}{s^2}+\frac{1}{(s-1)^2}-\frac{1}{s(s-1)}\right)=0,
\]
which involves the Schwarzian derivative $\{s,x\}$ given by
\[
\{s,x\}=\frac{s'''}{s'}-\frac{3}{2}\left(\frac{s''}{s'}\right)^2.
\]
The inverse map determines a map from the complex upper half plane into a fundamental domain given by a half-infinite strip in the complex upper half plane minus a semicircle of radius half the width of the strip and centred midpoint on the boundary of the strip on the real axis. See Figure one in Chapter VII of \cite{acia} for a diagrammatic suggestion of the domain. 
Perhaps less is well known is that the combinations
\begin{align}\label{411maps}
p_1=-4w_1-w_2-w_3,\nonumber\\
p_2=-w_1-4w_2-w_3,\\
p_3=-w_1-w_2-4w_3,\nonumber
\end{align}
for the same $w_1$, $w_2$ and $w_3$ also give solutions to Chazy's equation with $y=p_1$ or $p_2$ or $p_3$, due to the equilateral symmetry of the domain of the triangle with angles $(0,0,0)$. The proof of Theorem \ref{thm1} comes from the observation of this fact. 
The maps (\ref{411maps}) give a linear invertible transformation from $(w_1,w_2,w_3)$ to $(p_1,p_2,p_3)$, with $y=p_0=\frac{1}{3}(p_1+p_2+p_3)=-2(w_1+w_2+w_3)$ again a solution to (\ref{chazy}). Inverting the maps (\ref{411maps}), we can express the Darboux-Halphen system (\ref{dh}) in terms of $(p_1,p_2,p_3)$ that are solutions of Chazy's equation (\ref{chazy}), to obtain the first order system:
\begin{align}\label{c1n}
p_1'-\frac{1}{6}p_1^2 &=\frac{8}{27}(p_2-p_1)(p_1-p_3)-\frac{1}{54}(p_2-p_3)^2,\nonumber\\
p_2'-\frac{1}{6}p_2^2 &=\frac{8}{27}(p_3-p_2)(p_2-p_1)-\frac{1}{54}(p_3-p_1)^2,\\
p_3'-\frac{1}{6}p_3^2 &=\frac{8}{27}(p_1-p_3)(p_3-p_2)-\frac{1}{54}(p_1-p_2)^2.\nonumber
\end{align} 
The upshot of this move is that we have expressed the Darboux-Halphen system (\ref{dh}) as a nonlinear system of first-order ODE, with each of $p_1$, $p_2$, $p_3$ satisfying Chazy's equation. The solutions to (\ref{chazy}) are now permuted by the cyclic group of three elements $C_3$.
 Now, it is also very well-known that Chazy's equation (\ref{chazy}) can be written as the Ramanujan system (\ref{rde0}) 
\begin{align*}
p_1'&=\frac{1}{6}(p_1^2-q_1),\\
q_1'&=\frac{2}{3}(p_1q_1-r_1),\\
r_1'&=p_1r_1-q_1^2,
\end{align*}  
with $y=p_1$.
Comparing this to the system of differential equations given in (\ref{c1n}) we can solve for $q_1$ and $r_1$ in terms of $p_1$, $p_2$ and $p_3$ to get
 \[
 q_1=\frac{16}{9}(p_2-p_1)(p_3-p_1)+\frac{1}{9}(p_2-p_3)^2
 \] 
 and
 \[
r_1=\frac{1}{27}(2 p_1-p_2-p_3)(32(p_3-p_1)(p_2-p_1)-(p_2-p_3)^2).
 \]
Similarly, we obtain
\begin{align*}
 q_2&=\frac{16}{9}(p_3-p_2)(p_1-p_2)+\frac{1}{9}(p_3-p_1)^2,\\
  q_3&=\frac{16}{9}(p_1-p_3)(p_2-p_3)+\frac{1}{9}(p_1-p_2)^2,
\end{align*}
for $q_2=-6(p_2'-\frac{1}{6}p_2^2)$ and $q_3=-6(p_3'-\frac{1}{6}p_3^2)$,
and
\begin{align*}
r_2&=\frac{1}{27}(2 p_2-p_3-p_1)(32(p_1-p_2)(p_3-p_2)-(p_3-p_1)^2),\\
r_3&=\frac{1}{27}(2 p_3-p_1-p_2)(32(p_2-p_3)(p_1-p_3)-(p_1-p_2)^2),
\end{align*}
for $r_2=-\frac{3}{2}(q_2'-\frac{2}{3}p_2 q_2)$ and $r_3=-\frac{3}{2}(q_3'-\frac{2}{3}p_3 q_3)$.
We now pose the following question: Given the triple $(p_1,q_1,r_1)=(P(x),Q(x),R(x))$ satisfying Ramanujan's differential equations (\ref{rde}), can we solve for $(p_2,q_2,r_2)$ and $(p_3,q_3,r_3)$ in terms of $P$, $Q$ and $R$?
The answer is yes and to see this we first substitute the values $(p_1,q_1,r_1)=(P,Q,R)$ into the formula for $q_1=Q$ and $r_1=R$ above, and let $V=p_2+p_3$ and $u=p_2-p_3$.
Then we find that 
 \[
 9Q=16P^2-16 P V+4 V^2-3u^2
 \] 
 and
 \[
 R=\frac{1}{27}(2P-V)(8Q-3u^2).
 \]
 This gives an algebraic relation for $u$ and $V$ involving $P$, $Q$ and $R$.
 Solving this for $u$ and $V$ gives 
 \begin{align*}
 V&=2P+\frac{3}{2}T^{\frac{1}{3}}+\frac{3}{2}\frac{Q}{T^{\frac{1}{3}}},\\
 u&=\pm \sqrt{3}(\frac{Q^2}{T^{\frac{2}{3}}}+Q+T^{\frac{2}{3}})^{\frac{1}{2}},
 \end{align*}
 where $T=R+\sqrt{R^2-Q^3}$.
 Let $v=\frac{3}{2}T^{\frac{1}{3}}+\frac{3}{2}\frac{Q}{T^{\frac{1}{3}}}$ so that $V=2P+v$. Substitute these formulas for $u$ and $V$ back into the expressions for $(p_2,q_2,r_2)$ and $(p_3,q_3,r_3)$. Consequently, we find that given $(P,Q,R)$ satisfying Ramanujan's differential equations (\ref{rde}), we obtain
 \begin{align*}
 &(p_2,q_2,r_2)=\\
 &\left(P+\frac{u+v}{2},\frac{8}{9}u(u+v)+\frac{1}{36}(v-u)^2,\frac{1}{54}(3u+v)(16u(u+v)-(\frac{v-u}{2})^2)\right),\\
& (p_3,q_3,r_3)=\\
 &\left(P+\frac{v-u}{2},\frac{8}{9}u(u-v)+\frac{1}{36}(v+u)^2,\frac{1}{54}(v-3u)(16u(u-v)-(\frac{u+v}{2})^2)\right),
 \end{align*}
as further solutions to the Ramanujan system.
Furthermore, as $p_0=\frac{1}{3}(p_1+p_2+p_3)$ is also a solution to Chazy's equation, we find that the triple $(p_0, q_0, r_0)$ given by
\begin{align}\label{sum}
p_0&=P+\frac{1}{2}T^{\frac{1}{3}}+\frac{1}{2}\frac{Q}{T^{\frac{1}{3}}},\nonumber\\
q_0&=\frac{3}{2}Q+\frac{5}{4}T^{\frac{2}{3}}+\frac{5}{4}\frac{Q^2}{T^{\frac{2}{3}}},\\
r_0&=\frac{11}{4}R+\frac{21}{8}QT^{\frac{1}{3}}+\frac{21}{8}\frac{Q^2}{T^{\frac{1}{3}}},\nonumber
\end{align} 
also satisfies Ramanujan's differential equations (\ref{rde0}).
\end{proof}
Now suppose we are given the triple $(p_0,q_0,r_0)$ instead as a  triple satisfying Ramanujan's differential equations (\ref{rde0}). We might be naturally inclined to invert the equations given by (\ref{sum}) to recover $(P,Q,R)$. Indeed we can compute the inverse map which is amazingly given by
\begin{align*}
P&=p_0+\frac{1}{2}t^{\frac{1}{3}}+\frac{1}{2}\frac{q_0}{t^{\frac{1}{3}}},\\
Q&=\frac{3}{2}q_0+\frac{5}{4}t^{\frac{2}{3}}+\frac{5}{4}\frac{q_0^2}{t^{\frac{2}{3}}},\\
R&=\frac{11}{4}r_0+\frac{21}{8}q_0t^{\frac{1}{3}}+\frac{21}{8}\frac{q_0^2}{t^{\frac{1}{3}}},
\end{align*} 
where $t=r_0+\sqrt{r_0^2-q_0^3}$.
This solution is precisely the one obtained from the addition formula $\frac{1}{3}(p_1+p_2+p_3)$ with $p_1=p_0$
and the analogous solutions for $p_2$ and $p_3$ obtained from the triple $(p_0,q_0,r_0)$. Thus there seems to be a duality between the elements $(P,Q,R)\leftrightarrow(p_0,q_0,r_0)$. What this entails for the solutions of (\ref{rde}) given by quasi-modular form $E_2$ and modular forms $E_4$, $E_6$ remains to be investigated.

\section{Applications}
It would be interesting to investigate the applications of Theorem \ref{thm1}, for example, to see if we can deduce results obtained in the book \cite{ell}. We now assume that we know a solution of (\ref{rde}) given by
\begin{align*}
P&=1-24\sum^{\infty}_{n=1}\si_1(n)q^n=E_2,\\
Q&=1+240\sum^{\infty}_{n=1}\si_3(n)q^n=E_4,\\
R&=1-504\sum^{\infty}_{n=1}\si_5(n)q^n=E_6.
\end{align*}
Here $E_2$ is a quasi-modular form given by the Eisenstein series of weight 2, while $E_4$ and $E_6$ are modular forms given by the Eisenstein series of weight 4 and 6 respectively. The functions here involve $\si_1(n)$ the sum of divisor function, $\si_3(n)$ the sum of cube of divisor function and $\si_5(n)$ the sum of fifth powers of divisor function. Also $q=e^{2 \pi i x}$ is the nome, and with $\der q=2 \pi i \der x$, this gives
\[
\frac{\der}{\der x}=2\pi i q\frac{\der}{\der q}
\]
as a change of variable, so that the Ramanujan system (\ref{rde}) can be rewritten as
\begin{align*}
\pi i q\frac{\der}{\der q}P&=\frac{1}{12}(P^2-Q),\\
\pi i q\frac{\der}{\der q}Q&=\frac{1}{3}(PQ-R),\\
\pi i q\frac{\der}{\der q}R&=\frac{1}{2}(PR-Q^2).
\end{align*}
Now given $(p_1,q_1,r_1)=(P,Q,R)$, we compute and find that
\begin{align*}
p_2=4-96q^4-288q^8-384q^{12}-672q^{16}-\ldots=4P(q^4),\\
p_3=1+24q-72q^2+96q^{3}-168q^4+144q^{5}-\ldots=P(-q),
\end{align*}
with $p_0=\frac{1}{3}P(q)+\frac{4}{3}P(q^4)+\frac{1}{3}P(-q)=2P(q^2)$.
In the case for $p_2$, we see that identifying $\tilde q=q^4$ gives back the solution $(P,Q,R)$ to the differential equations (\ref{rde}) with an appropriate constant rescaling and likewise for the case $p_3$, the variable to be identified is $\tilde q=-q$. Similarly, in the case of $p_0$, identifying $\tilde q=q^2$ gives us back a constant rescaling of $(P,Q,R)$. 
Taking the triple
\begin{align*}
(p_0,q_0,r_0)=(2P(q^2),4Q(q^2),8R(q^2))
\end{align*} 
satisfying (\ref{rde0}), we can apply Theorem \ref{thm1} to iterate the process and get the triples
\begin{align*}
(p_4,q_4,r_4)&=(8P(q^8),64Q(q^8),512R(q^8)),\\
(p_5,q_5,r_5)&=(2P(-q^2),4Q(-q^2),8R(-q^2)),
\end{align*} 
satisfying Ramanujan's differential equations with $\frac{1}{3}(p_0+p_4+p_5)=p_2$.
This addition property holds in general because $\si_1(4n)+2\si(n)=3\si_1(2n)$ holds for all natural numbers $n$.

It is also interesting to deduce results involving the hypergeometric function ${}_2F_{1}\left(\frac{1}{2},\frac{1}{2};1;x\right)$. This assumes that we know that the above solutions given by the Eisenstein series can be expressed in terms of Jacobi's theta functions. Firstly, the formulas for the Schwarz triangle function $s=s(0,0,0,x)$ and its derivative give
\[
s=\frac{w_1-w_3}{w_2-w_3}=\frac{p_1-p_3}{p_2-p_3}=\frac{\frac{u-v}{2}}{u}=\frac{1}{2}-\frac{1}{2}\frac{v}{u}
\]
and
\[
s'=\frac{2}{3}s(p_2-p_1)=\frac{1}{3}s(u+v)=\frac{1}{6u}(u-v)(u+v).
\]
The inversion formula for $s(x)$ gives $x=\frac{z_2(s)}{z_1(s)}$ where $z_1$, $z_2$ are two linearly independent solutions of the hypergeometric differential equation
\[
s(1-s)z_{ss}+(1-2s)z_{s}-\frac{1}{4}z=0. 
\]
We shall take $z_1(s)={}_2F_{1}(\frac{1}{2},\frac{1}{2};1;s)$.
Furthermore, we have
\[
\frac{\der}{\der x}=\frac{z_1^2}{W}\frac{\der}{\der s},
\]
which gives
\[
s'=\frac{z_1^2}{W}
\]
where $W=c_0s^{-1}(s-1)^{-1}$ is the Wronskian of $z_1$ and $z_2$ and $c_0$ is a constant. 
From these formulas for $s$ and $s'$, we obtain
\[
z_1^2=-\frac{2}{3}c_0u.
\]
Suppose now that
\begin{align*}
Q&=\frac{1}{2}(a^8+b^8+c^8),\\
R&=\frac{1}{2}(b^{12}+c^{12}-3a^8(b^4+c^4))
\end{align*}
with $b^4=a^4+c^4$. Here $a=\theta_2(q)$, $b=\theta_3(q)$ and $c=\theta_4(q)$ are given by Jacobi's theta functions. 
Making the substitution, we find through eliminating $c^4$ that
\begin{align*}
v&=-\frac{3}{2}a^4-\frac{3}{2}b^4,\\
u&=3a^2 b^2.
\end{align*}
Substituting these into the formulas for $s$ and $s'$ give 
\[
s=\frac{1}{2}+\frac{a^4+b^4}{4a^2b^2}
\]
and
\[
s'=-\frac{(a^2+b^2)^2(a-b)^2(a+b)^2}{8a^2 b^2}.
\]
Together this gives
\[
z_1(s)^2=-2 c_0 a^2 b^2,
\]
where
\[
s=\frac{1}{2}+\frac{a^4+b^4}{4a^2b^2}.
\]
Now the hypergeometric function ${}_2F_{1}(\frac{1}{2},\frac{1}{2};1;x)$ is related to ${}_2F_{1}(\frac{1}{4},\frac{3}{4};1;x)$ through a quadratic transformation in 2 different ways. Specifically, we have 
\[
{}_2F_{1}\left(\frac{1}{2},\frac{1}{2};1;x\right)=(1-2x)^{-\frac{1}{2}}{}_2F_{1}\left(\frac{1}{4},\frac{3}{4};1;1-\frac{1}{(2x-1)^2}\right)
\]
and
\[
{}_2F_{1}\left(\frac{1}{2},\frac{1}{2};1;x\right)=(1-\frac{x}{2})^{-\frac{1}{2}}{}_2F_{1}\left(\frac{1}{4},\frac{3}{4};1;\frac{x^2}{(2-x)^2}\right).
\]
If we let $\kappa=\frac{a^2}{b^2}$ (this $\kappa$ is used to define the elliptic modulus), then using these 2 quadratic transformations we find
\[
{}_2F_{1}\left(\frac{1}{2},\frac{1}{2};1;\frac{1}{2}+\frac{\kappa}{4}+\frac{1}{4\kappa}\right)=\kappa^{\frac{1}{2}} {}_2F_{1}\left(\frac{1}{2},\frac{1}{2};1;1-\kappa^2\right),
\]
or that
\[
z_1(s)^2=-2c_0a^2 b^2=\frac{a^2}{b^2}\left({}_2F_{1}\left(\frac{1}{2},\frac{1}{2};1;1-\kappa^2\right)\right)^2.
\]
This gives a relationship
\begin{align*}
-2 c_0 b^4=\left({}_2F_{1}\left(\frac{1}{2},\frac{1}{2};1;\frac{b^4-a^4}{b^4}\right)\right)^2=\left({}_2F_{1}\left(\frac{1}{2},\frac{1}{2};1;\frac{c^4}{b^4}\right)\right)^2,
\end{align*}
between the the theta function $b$ and the hypergeometric function $z_1$ as a function of the square of the complementary modulus $1-\kappa^2$.
If we made an alternative substitution eliminating $a$ instead to determine $u$, $v$ in terms of $b$ and $c$, we obtain
\begin{align*}
v&=\frac{3}{2}b^4+\frac{3}{2}c^4,\\
u&=3b^2 c^2.
\end{align*}
This gives instead the Schwarz triangle function
\[
s=\frac{1}{2}+\frac{b^4+c^4}{4b^2c^2}
\]
and
\[
s'=-\frac{(b^2+c^2)^2(b-c)^2(b+c)^2}{8b^2 c^2},
\]
and for this formula for $s$ we find
\begin{align*}
z_1(s)^2=-2c_0b^2 c^2=\frac{c^2}{b^2}\left({}_2F_{1}\left(\frac{1}{2},\frac{1}{2};1;1-\frac{c^4}{b^4}\right)\right)^2,
\end{align*}
which gives
\[
-2c_0b^4=\left({}_2F_{1}\left(\frac{1}{2},\frac{1}{2};1;\frac{a^4}{b^4}\right)\right)^2=\left({}_2F_{1}\left(\frac{1}{2},\frac{1}{2};1;\kappa^2\right)\right)^2.
\]
The identification of this with the result that relates the theta function $b$ to the hypergeometric function $z_1$ as a function of the square of the modulus $\kappa^2$, 
\[
b^2={}_2F_{1}\left(\frac{1}{2},\frac{1}{2};1;\kappa^2\right),
\]
which is known to Jacobi and Ramanujan (see \cite{ell}), requires us to take $c_0=-\frac{1}{2}$. 

\section{Part Two}
The second part of the paper concerns the generalised Chazy equation with parameter $k=\frac{3}{2}$.
Using the same notation as in part one, we say that the triple of functions $(p(x),q(x),r(x))$ of the variable $x$ satisfies the non-linear system of differential equations associated to the generalised Chazy equation with parameter $k$ if the following set of equations are satisfied for the functions $p(x)$, $q(x)$ and $r(x)$ in the triple:
\begin{align*}\label{ndek}
\frac{\der p}{\der x}&=\frac{1}{6}(p^2-q),\nonumber\\
\frac{\der q}{\der x}&=\frac{2}{3}(pq-r),\\
\frac{\der r}{\der x}&=pr+\frac{k^2}{36-k^2}q^2.\nonumber
\end{align*}
It can be seen that taking $y=p$ gives a solution to the generalised Chazy equation 
\begin{equation*}\label{gchazy}
y'''-2yy''+3(y')^2-\frac{4}{36-k^2}(6y'-y^2)^2=0
\end{equation*}
with parameter $k$. The generalised Chazy equation is introduced in \cite{chazy1}, \cite{chazy2} and studied more recently in \cite{co96}, \cite{acht} and \cite{ach}. 
When $k=\frac{3}{2}$ we have the following system of equations:
\begin{align}\label{nde1}
\frac{\der p}{\der x}&=\frac{1}{6}(p^2-q),\nonumber\\
\frac{\der q}{\der x}&=\frac{2}{3}(pq-r),\\
\frac{\der r}{\der x}&=pr+\frac{1}{15}q^2.\nonumber
\end{align}
Taking $y=p$ then gives a solution to the generalised Chazy equation 
\begin{equation}\label{chazy32}
y'''-2yy''+3(y')^2-\frac{4}{36-(\frac{3}{2})^2}(6y'-y^2)^2=0
\end{equation}
with parameter $k=\frac{3}{2}$. We prove the following theorem, similar to Theorem \ref{thm1} in the first part. 
\begin{theorem}\label{thm2}
Suppose $(P(x),Q(x),R(x))$ satisfies the non-linear system of differential equations associated to the generalised Chazy equation with parameter $k=\frac{3}{2}$, i.e. we have
\begin{align*}
\frac{\der}{\der x}P&=\frac{1}{6}(P^2-Q),\nonumber\\
\frac{\der}{\der x}Q&=\frac{2}{3}(PQ-R),\\
\frac{\der}{\der x}R&=PR+\frac{1}{15}Q^2.\nonumber
\end{align*}
Let $Z=\frac{3R}{2Q}+\sqrt{-\frac{3}{5}Q}$ and $\bar{Z}=\frac{3R}{2Q}-\sqrt{-\frac{3}{5}Q}$.
Then the following holds. The triples
\begin{align*}
 (p_2,q_2,r_2)&=\left(P+Z,-\frac{5}{3}\bar{Z}^2,\frac{5}{9}\bar{Z}^2(2Z-\bar{Z})\right),\\
 (p_3,q_3,r_3)&=\left(P+\bar{Z},-\frac{5}{3}Z^2,\frac{5}{9}Z^2(2\bar{Z}-Z)\right),
 \end{align*}
also satisfy the system of differential equations (\ref{nde1}), while the triple
\begin{align*}
(p_0,q_0,r_0)&=\left(P+\frac{R}{Q},~ \frac{3}{5}Q-3\frac{R^2}{Q^2},~\frac{9}{5}R+3\frac{R^3}{Q^3}\right)
\end{align*} 
with $p_0$ given by $p_0=\frac{1}{3}(p_1+p_2+p_3)$, satisfies the non-linear system of differential equations associated to the generalised Chazy equation with parameter $k=3$, i.e. we have
\begin{align}\label{nde2}
\frac{\der}{\der x}p_0&=\frac{1}{6}(p_0^2-q_0),\nonumber\\
\frac{\der}{\der x}q_0&=\frac{2}{3}(p_0q_0-r_0),\\
\frac{\der}{\der x}r_0&=p_0r_0+\frac{1}{3}q_0^2.\nonumber
\end{align}
In this case $y=p_0$ satisfies the generalised Chazy equation with $k=3$.
Conversely, given a triple $(p_0,q_0,r_0)$ that satisfies the system of differential equations (\ref{nde2}) associated to the generalised Chazy equation with parameter $k=3$, we can find a triple $(P,Q,R)$ that satisfies the system of differential equations (\ref{nde1}) associated to the generalised Chazy equation with parameter $k=\frac{3}{2}$, where 
\begin{align*}
P&=p_0-\frac{r_0}{4z-q_0},\\
Q&=\frac{5}{3}z,\\
R&=\frac{5}{3}\frac{r_0 z}{4z-q_0},
\end{align*} 
and $z$ is a root of the cubic equation
\begin{align*}
16z^3-24q_0z^2+9q_0^2z-q_0^3-3r_0^2=0.
\end{align*}
\end{theorem}

\begin{proof}
Like in part one, the proof makes use of some facts about the generalised Chazy equation with parameter $k=\frac{3}{2}$. This generalised Chazy equation with parameter $k=\frac{3}{2}$ can be written in the symmetric form as a first-order system of differential equations
\begin{equation}\label{dh32}
w_1'=w_2w_3-w_1w_2-w_1w_3+\frac{16}{9}((w_1-w_2)(w_3-w_1)+(w_2-w_3)(w_1-w_2)+(w_3-w_1)(w_2-w_3)),
\end{equation}
with cyclic permutation, in analogy to the Darboux-Halphen system. Here
\begin{align*}
w_1&=-\frac{1}{2}\frac{\der}{\der x}\log\frac{s'}{s (s-1)},\\
w_2&=-\frac{1}{2}\frac{\der}{\der x}\log\frac{s'}{s-1},\\
w_3&=-\frac{1}{2}\frac{\der}{\der x}\log\frac{s'}{s},
\end{align*}
and $s$ is the Schwarz triangle function $s(\frac{2}{3},\frac{2}{3},\frac{2}{3},x)$ given by solutions of the differential equation
\[
\{s,x\}+\frac{5}{18}(s')^2\left(\frac{1}{s^2}+\frac{1}{(s-1)^2}-\frac{1}{s(s-1)}\right)=0.
\]
The solutions to the generalised Chazy equation (\ref{chazy32}) are given by the combinations
\begin{align}\label{411maps2}
p_1=-4w_1-w_2-w_3,\nonumber\\
p_2=-w_1-4w_2-w_3,\\
p_3=-w_1-w_2-4w_3,\nonumber
\end{align}
with $y=p_1$ or $p_2$ or $p_3$. If we consider $p_0=\frac{1}{3}(p_1+p_2+p_3)=-2(w_1+w_2+w_3)$, we find that this solves the generalised Chazy equation with parameter $k=3$ instead. 
In other words, $y=p_0$ solves the differential equation
\begin{equation*}\label{chazy3}
y'''-2yy''+3(y')^2-\frac{4}{36-3^2}(6y'-y^2)^2=0.
\end{equation*}
See also \cite{r16} for the occurrence of the Schwarz triangle function $s(\frac{2}{3},\frac{2}{3},\frac{2}{3},x)$ in the solutions to the generalised Chazy equation with $k=\frac{3}{2}$ and $k=3$.
The triangular domain again has equilateral symmetry with angles $(\frac{2\pi}{3},\frac{2\pi}{3},\frac{2\pi}{3})$. We also have the same linear invertible transformation from $(w_1,w_2,w_3)$ to $(p_1,p_2,p_3)$, and inverting the map (\ref{411maps2}), we can express the symmetric first-order system (\ref{dh32}) in terms of $(p_1,p_2,p_3)$ that are each solutions of the generalised Chazy equation with parameter $k=\frac{3}{2}$ to get the first order system:
\begin{align}\label{c2n}
p_1'-\frac{1}{6}p_1^2 &=\frac{5}{18}(p_2-p_3)^2,\nonumber\\
p_2'-\frac{1}{6}p_2^2 &=\frac{5}{18}(p_3-p_1)^2,\\
p_3'-\frac{1}{6}p_3^2 &=\frac{5}{18}(p_1-p_2)^2.\nonumber
\end{align}
The solutions are again permuted by the cyclic group of three elements $C_3$. As already mentioned, this generalised Chazy equation (\ref{chazy32}) can be written as the nonlinear system of first-order differential equations given by
\begin{align*}
p_1'&=\frac{1}{6}(p_1^2-q_1),\\
q_1'&=\frac{2}{3}(p_1q_1-r_1),\\
r_1'&=p_1r_1+\frac{1}{15}q_1^2,
\end{align*}  
with $p_1=y$.
Comparing this to (\ref{c2n}) and solving these differential equations for $q_1$, $r_1$ in terms of $p_1$, $p_2$ and $p_3$ gives
 \[
 q_1=-\frac{5}{3}(p_2-p_3)^2
 \] 
 and
 \[
r_1=-\frac{1}{3}q_1(2 p_1-p_2-p_3).
 \]
Similarly, we obtain
\begin{align*}
 q_2&=-\frac{5}{3}(p_3-p_1)^2,\\
  q_3&=-\frac{5}{3}(p_1-p_2)^2,
\end{align*}
for $q_2=-6(p_2'-\frac{1}{6}p_2^2)$ and $q_3=-6(p_3'-\frac{1}{6}p_3^2)$,
and
\begin{align*}
r_2&=-\frac{1}{3}q_2(2 p_2-p_3-p_1),\\
r_3&=-\frac{1}{3}q_3(2 p_3-p_1-p_2),
\end{align*}
for $r_2=-\frac{3}{2}(q_2'-\frac{2}{3}p_2 q_2)$ and $r_3=-\frac{3}{2}(q_3'-\frac{2}{3}p_3 q_3)$.
We again pose the problem, that given $(p_1,q_1,r_1)=(P,Q,R)$ a solution to the system of differential equations (\ref{nde1}) associated to the generalised Chazy equation with $k=\frac{3}{2}$, namely
\begin{align*}
P'&=\frac{1}{6}(P^2-Q),\\
Q'&=\frac{2}{3}(PQ-R),\\
R'&=PR+\frac{1}{15}Q^2,
\end{align*}  
can we solve for $(p_2,q_2,r_2)$ and $(p_3,q_3,r_3)$ in terms of $P$, $Q$ and $R$? The answer is again positive. Substitute $(p_1,q_1,r_1)=(P,Q,R)$ into the formula for $q_1$ and $r_1$ above, and we find that 
 \[
Q=-\frac{5}{3}(p_2-p_3)^2
 \] 
 and
 \[
 R=-\frac{1}{3}Q(2P-p_2-p_3).
 \]
 Solving this for $p_2$ and $p_3$ gives 
 \begin{align*}
 p_2&=P+\frac{3R}{2Q}\pm\frac{1}{2}\sqrt{-\frac{3}{5}Q},\\
 p_3&=P+\frac{3R}{2Q}\mp\frac{1}{2}\sqrt{-\frac{3}{5}Q}.
 \end{align*}
Let $Z=\frac{3R}{2Q}+\frac{1}{2}\sqrt{-\frac{3}{5}Q}$ and $\bar{Z}=\frac{3R}{2Q}-\frac{1}{2}\sqrt{-\frac{3}{5}Q}$.
The consequence of this is that given $(P,Q,R)$ solutions to the system of differential equations (\ref{nde1}), we obtain the triple
 \begin{align*}
& \left(P+Z,-\frac{5}{3}\bar{Z}^2,\frac{5}{9}\bar{Z}^2(2Z-\bar{Z})\right),\\
 &\left(P+\bar{Z},-\frac{5}{3}Z^2,\frac{5}{9}Z^2(2\bar{Z}-Z)\right),
 \end{align*}
as further solutions to the same system of differential equations (\ref{nde1}).
Furthermore $p_0=\frac{1}{3}(p_1+p_2+p_3)$ is now a solution to the generalised Chazy equation with $k=3$, and we find
\begin{align*}
p_0&=P+\frac{R}{Q},\\
q_0&=\frac{3}{5}Q-3\frac{R^2}{Q^2},\\
r_0&=\frac{9}{5}R+3\frac{R^3}{Q^3},
\end{align*} 
where $(p_0, q_0, r_0)$ satisfies the first-order system of differential equations associated to the $k=3$ generalised Chazy equation
\begin{align*}
p_0'&=\frac{1}{6}(p_0^2-q_0),\\
q_0'&=\frac{2}{3}(p_0q_0-r_0),\\
r_0'&=p_0r_0+\frac{1}{3}q_0^2.
\end{align*}  
Conversely, suppose now that we are given $(p_0,q_0,r_0)$ instead satisfying the first-order system associated to the $k=3$ generalised Chazy equation. We would like to invert the above map to recover $(P,Q,R)$ satisfying the first-order system associated to the $k=\frac{3}{2}$ generalised Chazy equation. We can compute the inverse map which is given by
\begin{align*}
P&=p_0-\frac{r_0}{4z-q_0},\\
Q&=\frac{5}{3}z,\\
R&=\frac{5}{3}\frac{r_0 z}{4z-q_0},
\end{align*} 
where $z$ satisfies the cubic equation
\begin{align*}
16z^3-24q_0z^2+9q_0^2z-q_0^3-3r_0^2=0.
\end{align*}
\end{proof}
The generalised Chazy equation with parameter $k=3$ is linearisable and this suggests a method of computing the solutions to the generalised Chazy equation with parameter $k=\frac{3}{2}$ from its linear counterpart. The relationship between the generalised Chazy equations with these two parameters was also investigated in \cite{r16}. There is also an interesting relationship involving the first-order system associated to the generalised Chazy equation with $k=2$ and the first-order system (\ref{nde2}) for the generalised Chazy equation with $k=3$. This boils down to the observation that for
\begin{align*}
w_1&=-\frac{1}{2}\frac{\der}{\der x}\log\frac{s'}{s (s-1)},\\
w_2&=-\frac{1}{2}\frac{\der}{\der x}\log\frac{s'}{s-1},\\
w_3&=-\frac{1}{2}\frac{\der}{\der x}\log\frac{s'}{s},
\end{align*}
where $s$ is the Schwarz triangle function $s(\frac{1}{3},\frac{1}{3},1,x)$ given by solutions of the differential equation
\[
\{s,x\}+\frac{4}{9}\frac{(s')^2}{s^2}=0,
\]
we have $y=-w_1-2w_2-3w_3$ and $y=-2w_1-w_2-3w_3$ satisfying the generalised Chazy equation with parameter $k=3$, while $y=-2w_1-2w_2-2w_3$ solves the generalised Chazy equation with parameter $k=2$. Inverting this linear map allows us to pass from a triple $(p,q,r)$ that is a solution to the differential equations (\ref{nde2}) to a triple $(\tilde p, \tilde q, \tilde r)$ that is a solution to the system of differential equations associated to the $k=2$ equation and vice-versa. 

Likewise, a similar relationship holds for the first-order system associated to the generalised Chazy equation with parameter $k=4$ and the first-order system associated to the generalised Chazy equation with parameter $k=3$, again due to the observation that for
\begin{align*}
w_1&=-\frac{1}{2}\frac{\der}{\der x}\log\frac{s'}{s (s-1)},\\
w_2&=-\frac{1}{2}\frac{\der}{\der x}\log\frac{s'}{s-1},\\
w_3&=-\frac{1}{2}\frac{\der}{\der x}\log\frac{s'}{s},
\end{align*}
where $s$ is the Schwarz triangle function $s(\frac{1}{3},\frac{1}{3},\frac{1}{2},x)$ given by solutions of the differential equation
\[
\{s,x\}+\frac{1}{2}(s')^2\left(\frac{\frac{8}{9}}{s^2}+\frac{\frac{3}{4}}{(s-1)^2}-\frac{\frac{3}{4}}{s(s-1)}\right)=0,
\]
we have $y=-w_1-2w_2-3w_3$ and $y=-2w_1-w_2-3w_3$ satisfying the generalised Chazy equation with parameter $k=3$, while $y=-2w_1-2w_2-2w_3$
solves the generalised Chazy equation with parameter $k=4$.

\end{document}